\documentclass[12pt]{amsart}

\usepackage{pdfsync}
\usepackage{amssymb}
\usepackage{stmaryrd}
\usepackage{fullpage}	% smaller margins
\usepackage{amscd}   	% for commutative diagrams
\usepackage[all]{xy} 		% for complicated commutative diagrams
\usepackage{comment} 	% for \begin{comment} ... \end{comment}
\usepackage{mathrsfs}      % for \mathscr
\newcommand{\marg}[1]{}
\newcommand{\david}[1]{}
\newcommand{\bjorn}[1]{}

\newcommand{\C}{{\mathbb C}}
\newcommand{\F}{{\mathbb F}}
\newcommand{\G}{{\mathbb G}}

\newcommand{\Q}{{\mathbb Q}}

\newcommand{\Z}{{\mathbb Z}}
\newcommand{\Qbar}{{\overline{\Q}}}

\newcommand{\calE}{{\mathcal E}}

\newcommand{\scrE}{{\mathscr E}}

\def\Q{\mathbb{Q}}
\def\C{\mathbb{C}}

\def\P{\mathbb{P}}

\def\Z{\mathbb{Z}}

\DeclareMathOperator{\Gal}{Gal}

\newcommand{\GalQ}{{\Gal}(\Qbar/\Q)}

\numberwithin{equation}{section}
\newtheorem{theorem}[equation]{Theorem}
\newtheorem{lemma}[equation]{Lemma}

\newtheorem{proposition}[equation]{Proposition}

\theoremstyle{definition}
\newtheorem{definition}[equation]{Definition}

\newtheorem{example}[equation]{Example}

\theoremstyle{remark}
\newtheorem{remark}[equation]{Remark}

\usepackage[
	backref,
	pdfauthor={David Brown}, % add other authors
]{hyperref}

\usepackage[alphabetic,backrefs,lite]{amsrefs} % for bibliography

\begin{document}

\title[Primitive integral solutions to]{Primitive integral solutions to $x^2 + y^3 = z^{10}$}
%\subjclass[2000]{Primary 55T10; Secondary 55-01}
%keywords{}
\author{David Brown}
\address{Department of Mathematics, University of California, Berkeley, CA
94720-3840, USA}
\email{brownda@math.berkeley.edu}

\date{\today}

\begin{abstract}
    We classify primitive integer solutions to $x^{2} + y^{3} =     z^{10}$. 
    The technique is to combine modular methods at the prime 5, number field enumeration techniques in place of modular methods at the prime 2, Chabauty techniques for elliptic curves
    over number fields, and local methods.

\end{abstract}

\maketitle

%\tableofcontents
%****************************************************************************
%****************************************************************************
%****************************************************************************
%****************************************************************************
\section{Introduction}
\label{introduction}

We say a triple $(s,t,u)$ is \emph{primitive} if $\gcd (s,t,u)=1$. For integers $a,b,c \geq 3$ the equation $x^{a} + y^{b} = z^{c}$ is expected to have no primitive integer solutions with $xyz \neq 0$, while for $n \leq 9$ the equation $x^{2} + y^{3} = z^{n}$ has many solutions with $xyz \neq 0$; the triple $(\pm 3, -2, 1)$ is, for instance, always a solution. Moreover, there are infinitely many parametriezed solutions with $xyz \neq 0$ and $z^2 \neq 1$ when $n \leq 5$ and there are some (but only finitely many) such solutions  for $6 \leq n \leq 9$; see \cite{pss} for the case $n = 7$ and a detailed review of previous work on generalized Fermat equations.

 The `next' equation to solve is $x^{2} + y^{3} = z^{10}$ -- it is the first of the form $x^{2} + y^{3} = z^{n}$ expected to have no solutions such that $xyz \neq 0$ and $z^2 \neq 1$.
\begin{theorem}
    \label{T:mainTheorem}
    The primitive integer solutions to $x^{2} + y^{3} = z^{10}$ are
    the 12 triples
    \[
        (\pm 1, -1, 0), (\pm 1, 0, \pm 1), (0,1,\pm 1), (\pm 3, -2, \pm 1).
    \]
\end{theorem}
It is clear that the only primitive solutions with $xyz = 0$ are the eight above. 
To ease notation we set $S(\Z) = \{(a,b,c) : a^2 + b^3 = c^{10}\, |\, (a,b,c) \text{ is primitive} \}$.\\

One idea is to use Edwards's parameterization of $x^{2} + y^{3} =
z^{5}$. His thesis \cite{Edwards:235} produces a list of 27 degree 12 polynomials $f_i \in \Z[x,y]$ such that if $(x,y,z)$ is such a
primitive triple, then there exist $i,s,t \in \Z$ such that
\[
z = f_i(s,t)
\]
and similar polynomials for $x$ and $y$. This would reduce the problem to
finding integral points on the 27 genus 5 hyperelliptic curves $-z^{2}
= f(s,t)$. This is a tempting approach, but the computational obstructions have not yet been overcome.

Alternatively, an elementary argument yields a parameterization of
$x^{2} + y^{3} = z^{2}$, leading one to (independently) solve each of the two equations
\begin{align*}
	y_{1}^{3} + y_{2}^{3} & = z^{5} \text{  or}\\
	y_{1}^{3} + 2y_{2}^{3} & = z^5.
\end{align*}
Following \cite{Bruin:twoCubes}, a resultant/Chabauty argument resolves the first equation. For the second equation one must pass to a degree 3 number field, where one runs into a genus 2 curve whose Jacobian has rank 3. There is less hope here for a classical Chabauty argument; \cite{Siksek:chabauty} gives a solution along these lines, combining classical Chabauty methods with the elliptic Chabauty methods of \cite{Bruin:ellChab}.

%  and the local information from the
% remaining cases will narrow down the curves from Edwards'
% parameterization to a managable set.

In a third direction one may consider the modular method used for
example in \cite{pss} to resolve the equation $x^{2} + y^{3} = z^{7}$.
This method is most effective for large $p$; in particular, for a prime $p \geq
7$ the modular curve $X(p)$ has genus $> 2$ and conjecturally for $p \geq 17$
the curves $X_E(p)$ have no non-trivial $\Q$-points. Stated as a question by Mazur \cite{Mazur:isogenies}*{P. 133}  for $p \geq 7$ and later relaxed to $p \geq 17$ (see \cite{FreyM:arithmeticModular}*{Table 5.3} for examples), this is now often called the Frey-Mazur conjecture. We note that direct application of this method to the equation $x^2 + y^3 = z^{10}$ at the prime $5$ fails because $X(5)$ has genus 0.\\

The approach of this paper is to combine the traditional modular methods  at the prime 5 described above and `elementary' modular methods (based on number field enumeration) at the prime 2. Here the relevant modular curve is $X(10)$, which has genus 13. However it covers an elliptic curve $X$ (with a modular interpertation) and the relevant twists $X_{(E,E')}(10)$ cover $X$ over a degree $6$ number field $K_{E,E'}$. For the pair $(E,E')$ corresponding to the solution $(3,-2,1)$ the elliptic curve $X$ has rank 1 over $K_{E,E'}$ and one may apply the elliptic Chabauty methods of \cite{Bruin:ellChab}. Various local methods finish off the other cases.

The computer algebra package Magma \cite{Magma} is used in an essential way throughout.\\

%****************************************************************************
\subsection*{Acknowledgements}
I thank Bjorn Poonen for numerous conversations and for carefully reading earlier drafts of this paper and the participants of \url{mathoverflow.net} for help finding various references. Some computations were done on \url{sage.math.washington.edu}, which is supported by National Science Foundation Grant No. DMS-0821725.

%****************************************************************************

%****************************************************************************
%****************************************************************************

\section{A modular quotient of X(10)}
\label{S:relevantCurve}
%****************************************************************************
%****************************************************************************
Here we construct an elliptic quotient $X$ of the full modular curve $X(10)$ whose twists $X_E$ will be the center of our calculations.\\

Let $E$ be an elliptic curve and $p$ a prime. Following \cite{pss}*{Section 4}, we define $X_E(p)$ to be the compactified moduli space of elliptic curves $E'$ plus symplectic isomorphisms (i.e. respecting Weil pairings) $E'[p] \cong E[p]$ of $\GalQ$-modules. Similarly, we define $X(p)$ to be the variant of the classical modular curve which parameterizes $E'$ plus symplectic isomorphisms  $E'[p] \cong \mu_p \times \Z/p\Z$ of $\GalQ$-modules; in particular $X(p)$ is defined over $\Q$, is geometrically connected, and $X_E(p)_{\C}$ is isomorphic to $X(p)_{\C}$. 
Finally, for $p = 5$ we define (as in \cite{pss}*{4.4}) the variant $X^-_E(5)$ to be the compactified moduli space of elliptic curves $E'$ plus anti-symplectic isomorphisms $E'[5] \cong E[5]$ of $\GalQ$-modules; here we define anti-symplectic to mean that the map induced by Weil pairings $\mu_5 \to \mu_5$ is the map $\zeta \mapsto \zeta^2$.

\begin{remark}[\cite{pss}*{4.4}]
\label{R:anti}
  Any isomorphism $E'[5] \cong E[5]$ of $\GalQ$-modules gives rise to a point on either $X_E(5)$ or $X_E^-(5)$. Indeed, for a positive integer $N$ with $(N,5) = 1$, the multiplication by $N$ map $E \xrightarrow{[N]} E$ induces an isomorphism $E[5] \cong E[5]$ which changes the Weil pairing by $N^2$. Since $\F_5^* / (\F_5^*)^2 = \{1,2\}$, after composing with $[N]$ for some $N$, we arrive at an isomorphism $E'[5] \cong E[5]$  which either respects the Weil pairing or changes it by 2.
\end{remark}

Recall that $X_0(p)$ is the modular curve whose points correspond to $p$-isogenies of elliptic curves up to twists. There are natural maps 
\[
X(p) \to X_0(p) \to X(1).
\]
When $p = 2$, $\Gal(X(2) / X(1)) \cong S_3$, and one can check by a direct calculation that the quotient of $X(2)$ by the
normal subgroup $A_3$ is the degree 2 cover $X_{\Delta} \cong \P^1$ of $X(1) \cong \P^1$ given by $z \mapsto z^2 + 12^3$.

\begin{definition}\label{D:X}
  We define $X$ to be the normalization of $X_{\Delta} \times_{X(1)} X_0(5)$. Denote by $Y'$ the affine curve $Y(1)-\{12^3\}$, let  $Y \subset X$ be the preimage of $Y'$ in $X$, and let $K$ be a number field; then a point in $Y(K)$ corresponds (up to twists) to a 5-isogeny $E \to E'$ defined over $K$ with $j(E') \neq 12^3$ and a choice of a square root of $j(E)-12^3 = c_6^2/\Delta_E$ in $K$ (i.e. the fiber product $X_{\Delta} \times_{X(1)} X_0(5)$ is smooth away from cusps and the fiber of $12^3 \in X(1)$, which one can see via moduli using \cite{DiamondI:modularCurves}*{Equation 9.1.2} or directly from the equations for $X$ computed below). 
\end{definition}

	\[
	\xymatrix{
		X(10) 	\ar[rr]\ar[dd]\ar@{-->}[dr]	&	 		&	X(5)	\ar[d]	\\
		 		&	X 	\ar[r]\ar[d]	&	X_0(5)	\ar[d]	\\		  
		X(2) 	\ar[r]		&	X_{\Delta} 	\ar[r]		&	X(1)	
	}
	\]

Since $A_3$ is normal in $S_3$, the natural map $X_E(2) \to X(1)$ factors
through a twist $X_{\Delta_E}$ of $X_{\Delta}$ mapping to $X(1)$ via
$z \mapsto \Delta_E\cdot z^2 + 12^3$,  \marg{This is correct up to a square} and we define
$X_E$ to be the normalization of $X_{\Delta_E} \times_{X(1)} X_0(5)$. An easy calculation using \cite{DiamondI:modularCurves}*{Equation 9.1.2} shows that $X_E$ has genus 1; we omit the proof as this will be clear below when we compute explicit equations for $X_E$.
Again, away from the cusps  and the fiber over $12^3 \in X(1)$, the fiber product is smooth and so for a number field $K$ and for $Y_E$ the preimage of $Y'$ in $X_E$, a point in $Y_E(K)$ corresponds (up to twists) to a 5-isogeny $E' \to E''$ defined over $K$ such that $j(E) \neq 12^3$ and a choice of a square root of $(j(E')-12^3)/\Delta_{E} = c_6^2/(\Delta_E\Delta_{E'})$ in $K$.

%-----------------------------------------------------------------------------------------------
\subsection{Equations}
\label{ss:equations}
%-----------------------------------------------------------------------------------------------
By \cite{McMurdy:explicitParam}*{Table 3}, equations for the map $Y_0(5) \xrightarrow{\pi_1} Y(1)$ sending a 5-isogeny $(E \to E')$ to $j(E)$ are given by
\[
t \mapsto \frac{(t^2 + 250t + 3125)^3}{t^5} = \frac{(t^2 - 500t - 15625)^2(t^2 + 22t + 125)}{t^5} + 12^3.
\]
Equations for $Y_E$ are thus obtained by setting equal the equations
\[
\Delta_E\cdot z^2 + 12^3 = \frac{(t^2 - 500t - 15625)^2(t^2 + 22t + 125)}{t^5} + 12^3,
\]
simplifying, and making the  change of coordinates $y = zt^3/(t^2 - 500t - 15625)$. This gives
\[
\Delta_E\cdot y^2 = t(t^2 + 22t + 125),
\]
and the map $Y_E \xrightarrow{p_{\Delta_E}} X_{\Delta_E}$ is given by 
\[
(t,y) \mapsto \frac{y(t^2 - 500t - 15625)}{t^3}.
\]

%****************************************************************************
%****************************************************************************
\section{Modular Methods at 2}
\label{S:modular2}
%****************************************************************************
%****************************************************************************
Following \cite{pss}*{4.6} we define for a primitive triple $(a,b,c)$ the elliptic curve
\[
E = E_{(a,b,c)} : Y^2 = X^3 + 3bX - 2a.
\]

\begin{remark}
\label{R:modSquares}
The elliptic curve $E$ has $j$-invariant $12^3b^3/c^{10} = -12^3a^2/c^{10} + 12^3$; thus $j(E) - 12^3 = -12^3a^2/c^{10}$, which is $-3$ times a square. 
\end{remark}

Setting $E_0 := E_{(3,-2,1)}$, this remark proves the following.  

\begin{lemma}
\label{L:moduli2}
  To $(a,b,c) \in S(\Z)$ one may associate a point on $X_{\Delta_{E_0}}(\Q)$. 
\end{lemma}

In the following lemmas, we calculate $E_{(a,b,c)}[2]$ as a $\GalQ$-module. The conclusion will be that the set of such possibilities is finite and explicitly computable. In fact, one can realize each possible Galois module as $E'[2]$ for $E'$ an elliptic curve, with finitely many explicit possibilities for $E'$.  One may thus associate to $E_{(a,b,c)}$ a point on one of the modular curves $X_{E'}(2)$ and then combine this with level lowering at the prime $5$ to get a point on a twist of the genus 13 curve $X(10)$. Another idea is to use the explicit equations for the curve $X_{E'}(2)$ given in \cite{RubinS:mod2} to derive local information about the triple $(a,b,c)$.

The idea is that knowledge of $E_{(a,b,c)}[2]$ is equivalent to knowledge of the splitting field $L$ of the polynomial $f = X^3 + 3bX - 2a$, and we will find that $L$ is unramified outside of $\{2,3\}$. By Hermite's theorem, there are only finitely many fields of degree at most 3 and unramified outside of $\{2,3,\infty\}$, and with the aid of a computer one can easily enumerate them and recognize each as the field of definition of the 2-torsion of an elliptic curve $E'$ with good reduction outside of $\{2,3\}$, effectively `lowering the level' of $E_{(a,b,c)}$ at the prime $2$ (an otherwise difficult task given that the usual level lowering theorems don't apply -- the mod $2$ representation is often ramified at 2). The details come in the next two lemmas.

\begin{lemma}[`Level lowering' at 2]
\label{L:reducible2}
%   Let $(a,b,c) \in S(\Z)$ be a primitive triple and suppose that $f(x) = x^3 + 3bx - 2a$ is reducible. Then there is an isomorphism $\Q(E_{(a,b,c)}[2]) \cong \Q(E_{(0,1,1)})$.
  Let $(a,b,c) \in S(\Z)$ be a primitive triple and suppose that $a \neq 0$. Then $f(x) = x^3 + 3bx - 2a$ is irrreducible.
\end{lemma}

\begin{proof}
  Let $K = \Q(E_{(a,b,c)}[2])$ be the splitting field of $f$. The discriminant of $E_{(a,b,c)}$ is $-12^3c^{10}$. By the standard Tate uniformization argument (that after base change to the maximal unramified extension $\Q_p^{\text{un}}$ of $\Q_p$
there exists an analytic Galois equivariant isomorphism of $E_{(a,b,c)}\otimes \Q_p^{\text{un}}$ with a Tate curve $\G_m/q^{\Z}$), $E_{(a,b,c)}[2]$ (and thus $K$) is unramified at a prime $p>3$ of multiplicative reduction if $2 | v_p(\Delta_{E_{(a,b,c)}})$ (see \cite{Ellenberg:finiteFlatness}*{Corollary 1.2} for a more detailed proof of this fact). Since  $v_p(\Delta_{E_{(a,b,c)}}) = v_p(c^{10})$ for $p > 3$, we conclude that $K$ is unramified outside of $\{2,3\}$. Since $f$ is reducible then $K$ has degree 1 or 2. There are only finitely many such fields of degree $\leq 2$, given by polynomials $x^2 + D$ with $D \in \{ \pm 1, \pm 2, \pm 3, \pm 6\}$. 

Let $E_D$ be the elliptic curve given by $y^2 = x(x^2 + D)$. Then there exists a $D$ as above and  a point $P \in X_{E_D}(2)(\Q)$ representing $E_{(a,b,c)}$ up to quadratic twist. Explicit equations paramaterizing such curves are given in \cite{RubinS:mod2}: there exist $u,v \in \Q$ such that $E_{(a,b,c)}$ is isomorphic to the curve $E_{u,v}\colon y^2 = x^3 + 3D(3v^2-Du^2)x - 2(9D^2uv^2 - D^3u^3)$. The given model has discriminant $\Delta(E_{u,v}) = -2^6 3^6 D (v(v^2 + Du^2)D)^2$. As this may only change by a $12^{\text{th}}$ power, one concludes that for $D \neq 3$, $j(E_{u,v})-12^3 = c_6(E_{u,v})^2/\Delta(E_{u,v})$ is not $-3$ times a square and thus (by  remark \ref{R:modSquares})  cannot be isomorphic  to $E_{(a,b,c)}$. 

When $D = 3$, further analysis of the equation $j(E_{u,v}) = j(E_{(a,b,c)}) = 12^3 b^3/c^{10}$ produces rational points on one of the genus 2 curves given by $y^2 = x^5 - 3^5$ and $y^2 = x^5 - 3^7$;  an application of Chabauty's method (recorded in the transcript of computations at \cite{me:2310Transcript}) determines the finite set of such points, and the only one corresponding to a primitive triple has $a = 0$.    
\end{proof}

\begin{lemma}
\label{L:level2}
  Let $(a,b,c) \in S(\Z)$ be a primitive triple. Suppose that $f(x) = x^3 + 3bx - 2a$ is irreducible. Then $\Q(E_{(a,b,c)}[2]) \cong \Q(E_{(3,-2,1)}[2])$.
\end{lemma}

\begin{proof}
  The proof is the same as lemma \ref{L:reducible2}, except that here  the computer algebra package Sage (which implements the Jones database of number fields \cite{Jones:tables})
is used to enumerate all degree 3 number fields unramified outside of $\{2,3,\infty\}$; a transcript of computations verifying  this can be found at \cite{me:2310Transcript}.

\end{proof}

\begin{remark}
 Lemmas \ref{L:reducible2} and \ref{L:level2} will be used in section \ref{S:local} to give local information about possible values of $j(E_{(a,b,c)}) = 12^3b^3/c^{10}$. 
\end{remark}

\begin{remark}
The  mod $3$ Galois representations arising from solutions to the equation $x^2 + y^3 = z^{15}$  are similarly ramified at the prime $3$, so that Ribet's level lowering theorem again does not apply. Nonetheless, the techniques of the `level lowering' lemmas \ref{L:reducible2} and \ref{L:level2} can be pushed (with more work and knowledge of the subfield structure of $\Q(E[3])$) to classify all mod 3 representations arising from solutions to this equation. It remains to be seen if the local information one obtains is actually useful in  solvng this Fermat equation.
\end{remark}

%****************************************************************************
%****************************************************************************
\section{Modular Methods at 5}
\label{S:modular5}
%****************************************************************************
%****************************************************************************
Let $(a,b,c) \in S(\Z)$, let $E = E_{(a,b,c)}$, and recall that we defined $E_0$ to be $E_{(3,-2,1)}$. Following \cite{pss}*{Section 6}, we classify the possibilities for the Galois module $E[5]$.
\begin{lemma}
\label{L:irreducible}
  $E[5]$ is irreducible as a $\GalQ$-module.
\end{lemma}
\begin{proof}
If $E[5]$ is reducible then there exists a 5-isogeny defined over $\Q$. One can thus associate to  $E$ a point on $X_0(5)(\Q)$. Together with lemma \ref{L:moduli2}, this implies that $E$ corresponds to a point on $X_{E_0}(\Q)$, which a Magma calculation reveals has rank 0. There are six torsion points and they have image $\{-102400/3, 20480/243, \infty\}$ in $X(1)$. For $(a,b,c) \in S(\Z)$,  $v_5(j(E_{(a,b,c)})) = v_5(12^3b^3/c^{10})$ is divisible by at least one of 3 or 10. On the other hand, for  $j = -102400/3$ or $20480/243$ one has $v_5(j) \in \{1,2\}$. We conclude that these do not correspond to j-invariants of elliptic curves coming from $(a,b,c) \in S(\Z)$.
\end{proof}

Let $\calE$ be the following set of $13$ elliptic curves over $\Q$ in the
notation of~\cite{Cremona:modularBook}:
\begin{gather*}
  \text{24A1}, \text{27A1}, \text{32A1}, \text{36A1}, \text{54A1}, 
  \text{96A1}, \text{108A1}, \text{216A1}, \text{216B1}, \text{288A1}, 
  \text{864A1}, \text{864B1}, \text{864C1}.
\end{gather*}

\begin{lemma}
\label{L:level5}
There exists an $E'' \in \calE$ and a quadratic twist $E'$ of $E$ such that $E'[5] \cong E''[5]$ as $\GalQ$-modules.
\end{lemma}

\begin{proof}
  This is identical to \cite{pss}*{Lemma 6.1}, with the remark that since $13$ is not a square mod $5$ one can again exclude the 14$^{\text{th}}$ newform.
\end{proof}

\begin{definition}
  For an elliptic curve $E$ with $j$-invariant $j$, define $K_E$ to be the number field $\Q(\alpha)$, with $\alpha$ a root of the polynomial
  \[
  f(t) = (t^2 + 250t + 3125)^3 - jt^5.
  \]
As the field $K_E$ only depends on $j(E)$, we will sometimes denote it by $K_j$.
\end{definition}
  By the explicit equations of  \ref{ss:equations},  an elliptic curve $E'$ has a $5$-isogeny over $K_{E'}$ and gives rise to a point on $X_0(K_{E'})$.
Moreover, by lemma \ref{L:level5}, for  $E$ corresponding to a primitive triple $(a,b,c) \in S(\Z)$ there exists an $E' \in \scrE$ such that $K_E = K_{E'}$ (since $E[5] \cong E'[5]$, $E$ has a 5-isogeny over a field $L$ if and only if $E'$ does). Thus, $E$ corresponds to a point in $X_0(5)(K_{E'})$. Combining this with lemma \ref{L:moduli2}, we arrive at the key observation.

  \begin{lemma}
    \label{L:key}
    For $(a,b,c) \in S(\Z)$, there exists an $E' \in \scrE$ such that $E_{(a,b,c)}$ corresponds to a point $P_E$ in $X_{E_0}(K_{E'})$ whose image in $X_{\Delta_{E_0}}(K_{E'})$ in fact lands in $X_{\Delta_{E_0}}(\Q)$.
  \end{lemma}

The image of $P_E$ in $X(1)$ (i.e. $j(E)$) is rational too. As the field $K_E$ depends only on the $j$-invariant of $E$, we note that the set of $j$-invariants of curves in $\calE$ is
\begin{gather*}
  \{35152/9, 0, 1728, 9261/8, 21952/9, -3072, -6, -216, -13824, 1536\}.
\end{gather*}

% {@ <35152 / 9, 24a1>, 
% <0, 27a1>, 
% <1728, 32a1>, 
% <0, 36a1>, <9261/8, 54a1>, 
% <21952/9, 96a1>, 
% <0, 108a1>, 
% <-3072, 216a1>, 
% <-6, 216b1>, 
% <1728, 288a1>, 
% <-216, 864a1>, 
% <-13824, 864b1>, 
% <1536, 864c1> 
% @}

%****************************************************************************
%****************************************************************************
\section{Elliptic Chabauty}
\label{S:ellChab}
%****************************************************************************
%****************************************************************************
We now study the situation of lemma \ref{L:key}.
\begin{proposition}
\label{P:ellChab}
  Let $j \in \{0,1728,-13824\}$. Then 
\[
j(X_{E_0}(K_j)) \cap X(1)(\Q) \in \{0, 1728, -13824, -102400/3, 20480/243, \infty\}.
\]
\end{proposition}

The proof requires consideration of  the following problem. Given an elliptic curve $E$ over $\Q$, a map $E \xrightarrow{\pi} \P^1$ defined over $\Q$, and a number field $K$ of degree $d > 1$ over $\Q$,  one would like to determine the subset of $E(K)$ mapping to $\P^1(\Q)$ under $\pi$. Let $r$ be the rank of $E(K)$. Suppose further that $r < d$; then under this hypothesis a partial solution to this problem has been worked out in  \cite{Bruin:ellChab}, using a method analogous to Chabauty's method (see \cite{McCallumP:chabautySurvey} for a survey) in that one expands the map $p$-adic analytically locally (i.e. in terms of $p$-adic power series) and uses Newton polygons to analyze the solutions. This method has been completely implemented in Magma; see \cite{Bruin:ellChab} for a succinct description of the method and instructions for use of its Magma implementation.

To use this we need to understand the output of the Magma function 
\[
\verb+Chabauty(MWmap, Ecov, p)+. 
\]
The first argument \verb+MWmap+ is a map from an abstract abelian group into the Mordell-Weil group of $E$ over $K$; we denote by $A$ its domain and $G$ its image. The second argument \verb+Ecov+ is a map from $E$ to $\P^1$ which is defined over $\Q$. The third argument \verb+p+ is a prime of good reduction for $E$ for which the map \verb+Ecov+ is also of good reduction. The function returns values $N,V,R$ and $L$ as follows (quoting the Magma documentation):
\begin{itemize}
\item $N$ is an upper bound for the number of points $P\ \in G$ such that $\verb+Ecov+(P) \in \P^1(\Q)$.
\item $V$ is a set of elements of $A$ that have images in $\P^1(\Q)$.
\item $R$ is a number such that, if $[E(K) : G]$ is finite and prime to $R$ then $N$ is also a bound for the number of points $P \in E(K)$ with image in $\P^1(\Q)$. 
\item $L$ is extra information which we will not use.
\end{itemize}
An important point is that one does not need to know the entire Mordell-Weil group $E(K)$, only a subgroup $G$ with index prime to $R$. In the following proof of  proposition \ref{P:ellChab}  we will not be able to compute all of $E(K)$.

\begin{proof}[Proof of proposition \ref{P:ellChab}] Magma code verifying the following can be found at \cite{me:2310Transcript}. Below, for $P \in X_{E_0}$ a cusp (so that it represents a degenerate elliptic curve) we write $j(P) = \infty$.\\

%--------------------------------------------------------------------------------------
Let $j = -13824$.
%--------------------------------------------------------------------------------------
The elliptic curve $E_0$ given by $y^2 = x^3 - 6x -6$ corresponds to the primitive triple $(3,-2,1) \in S(\Z)$. It has $j$-invariant $-13824$ and Cremona label 1728r1. It is a quadratic twist of the elliptic curve $E \in \calE$ with Cremona label 864b1, given by the equation $y^2 = x^3 - 24x - 48$. This is the most important case to consider in that there is actually a point on $X_{E_0}(K_E)$ corresponding to a triple $(a,b,c) \in S(\Z)$.

A Magma computation reveals that $X_{E_0}(K_E)$ has rank 1. One can construct explicitly the point $P \in X_{E_0}(K_E)$ corresponding to a $5$-isogeny of $E$ over $K_E$. The \verb+Chabauty+ routine returns $N = 8$, $\#V = 8$, and $R = 40$. The images under \verb+MWmap+ of $V$ are the known torsion points of lemma \ref{L:irreducible} and $\pm P$. Using Magma, one can check that the subgroup generated by $P$ and the torsion is $2$ and $5$ saturated, so that the index $[X_{E_0}(K_E) : G]$ is prime to $R$; the brute force point search necessary to check that $G$ generates the Mordell-Weil group is infeasible. The $j$-invariants of the (possibly degenerate) elliptic curves corresponding to these 8 points are $\{-13824, -102400/3, 20480/243, \infty\}$.\\

%--------------------------------------------------------------------------------------
Let $j = 0$.
%--------------------------------------------------------------------------------------
A Magma computation reveals that $X_{E_0}(K_j)$ has rank 1. Let $E \in \calE$ with $j(E) = 0$ (there are two such curves). One can construct explicitly the point $P \in X_{E_0}(K_j)$ corresponding to a $5$-isogeny of  $E$ over $K_j$. The \verb+Chabauty+ routine returns $N = 10$, $\#V = 10$ and $R = 2$. Using Magma, one can check that the subgroup generated by $P$ and the torsion points is $2$ saturated. The $j$-invariants of the (possibly degenerate) elliptic curves corresponding to the images in $X_{E_0}(K_j)$ under $\verb+MWmap+$ of $V$ are $\{0, -102400/3, 20480/243, \infty\}$.\\

%--------------------------------------------------------------------------------------
Let $j = 1728$.
%--------------------------------------------------------------------------------------
A Magma computation reveals that $X_{E_0}(K_j)$ has rank 0. The torsion subgroup has size 12, of which 4 points represent elliptic curves with $j$-invariant which is not a rational number. The $j$-invariants of the (possibly degenerate) curves they represent are $\{1728, -102400/3, 20480/243, \infty\}$.

\end{proof}

\begin{remark}
  We conclude that if $(a,b,c) \in S(\Z)$ is a primitive triple such that $E_{(a,b,c)}[5] \cong E[5]$ for an elliptic curve $E$ such that $j(E) \in \{0,1728,-13824\}$, then $(a,b,c)$ is one of the triples of theorem \ref{T:mainTheorem}.
\end{remark}

\begin{remark}
  For other values of $j$ one can compute that the rank of $X_{E_0}(K_j)$ is at most 3, but a brute force point search is too slow to explicitly determine a finite index subgroup.
\end{remark}

%****************************************************************************
%****************************************************************************
\section{Local Methods}
\label{S:local}
%****************************************************************************
%****************************************************************************

Here we use local methods inspired by \cite{pss}*{7.4} to exclude the existence of any further primitive triples. 
\begin{proposition}
   Let $E \in \calE$ and suppose $j(E) \not \in \{0, 1728, -13824\}$. Let $E'$ be an elliptic curve such that $E[5] \cong E'[5]$. Then $j(E') \neq j(E_{(a,b,c)})$ for any primitive triple $(a,b,c) \in S(\Z)$.
\end{proposition}

\begin{proof}
  MAGMA code verifying this (as described below) is available at \cite{me:2310Transcript}.
\end{proof}

\begin{remark}
  Note that this completes the proof of theorem \ref{T:mainTheorem}.
\end{remark}

The idea is the following.
For any morphism $X \xrightarrow{f} Y$ of varieties 
defined over $\Q$ and any prime $p$, one can (in principle) algorithmically  determine the image $f(X(\Q_p))$ of the $p$-adic points. The existence of such an algorithm follows from an effective elimination of quantifiers; see \cite{Macintyre:definable}. We explain how to do this for a map $\P^1 \to \P^1$ below.
By lemma \ref{L:level5} and remark \ref{R:anti}, $E_{(a,b,c)}$ gives rise to a point on either $X_E(5)$ or $X^-_E(5)$; the idea is then to apply this to the two maps $X_E(5) \to X(1)$ (resp. $X^-_E(5) \to X(1)$) and $S \to X(1)$.
Explicit equations for the map $X_E(5) \to X(1)$ (resp. $X^-_E(5) \to X(1)$) are given in \cite{RubinS:modpFamilies} (resp. the appendix), and the map $S \to X(1)$ 
        sends a primitive triple to the $j$-invariant $j(E_{(a,b,c)}) = 12^3b^3/c^{10}$.
For each $E \in \calE$ with $j(E) \not \in \{0,1728,-13824\}$, we will find a prime $p$ such
        that the $p$-adic images of these two maps do not intersect; here we of course restrict the domain of $S$ to primitive triples. Using the lemmas of section \ref{S:modular2}, one can also compare this to the local information coming from the maps $X_E(2) \to X(1)$.

Now we make this idea precise. Let $\P^1 \xrightarrow{\phi} \P^1$ be given by the pair of homogenous polynomials $f_1(s,t), f_2(s,t) \in \Z[s,t]$ and let $p$ be a prime number. We partition the set $\P^1(\Q_p)$ into the two residue classes $R_1 = [\Z_p:1]$ and $R_2 = [1:p\Z_p]$ and instead study the single variable polynomials $f_i(R_j) \in \Q_p[x]$ (where now $x$ will range over $\Z_p$). Let $f_i(R_j) = c_{ij} \cdot \prod_k (x - \alpha_{i,j,k})$. Using Newton polygons one can explicitly determine $\alpha_{i,j,k}$ to any desired precision, and from this is it straightforward to determine all values of $f_i(R_j)$ for $x \in \Z_p$.

\begin{example}
  As a very simple example, suppose $X_E(5) \xrightarrow{\phi} X(1)$ is given by $[f_1,f_2]$ and suppose that $v_p(\alpha_{i,1,k}) = 4/3$ for each
  $i,k$. Then for $x \in \Z_p$, $v_p(x-\alpha_{i,1,k}) = \min \{v_p(x), 4/3\} \in
  \{0,1,4/3\}$. Thus, on the residue class $[\Z_p:1]$, one has $v_p(f_i(x,1)) =
  v_p(c_{i,1}) \cdot \deg f_i \cdot \min \{v_p(x), 4/3\}$; setting $a_i =
  v_p(c_{i,1})$ we conclude that $\phi([\Z_p:1]) \subset [p^{a_1}\Z_p^{*} :
  p^{a_2}\Z_p^{*}]$. Suppose now that $p = 3$ and $a_1 = a_2$. Since $\gcd(b,c)
  = 1$, $v_3(j(E_{(a,b,c)})) = v_3(12^3 b^3/c^{10}) \neq 0$. We conclude that $v_3(j(E')) \neq v_3(j(E_{(a,b,c)}))$ for any $E' \in R_1 \subset X_E(5)(\Q_3)$ and $(a,b,c) \in S(\Z_p)$.
\end{example}

We have written MAGMA code (available at \cite{me:2310Transcript}) which takes as input an elliptic curve $E$ and a prime $p$ and returns, for each residue class $R_i$, a factorization $j = n\cdot \prod(x -\alpha_j)/\prod(x-\beta_j)$, where $j$ is the affine part (i.e. the quotient $f_1/f_2$) of the map $X_E(5) \to X(1)$ restricted to the residue class $R_i$. There is an optional parameter `anti'; when this is set to `true' the routine instead returns this data for the map $X^{-}_{E}(5) \to X(1)$. Using this, we ran the local test described above for each $E \in \calE$ with $j(E) \not \in \{0,1728,-13824\}$ and in each case found primes $p,p' \in \{2,3,5\}$ for which $X_E(5)$ (resp. $X^-_E(5)$) fails the local test at $p$ (resp. $p'$).

%****************************************************************************
%****************************************************************************
\section*{Appendix: Computing explicit equations for $X^{-}_{E}(5) \to X(1)$}
\renewcommand{\thesection}{A}
\renewcommand{\thesubsection}{\thesection}
\gdef\sectionname{Appendix}
\setcounter{equation}{0}
\setcounter{footnote}{0}
\makeatletter\@addtoreset{equation}{subsection}\makeatother
\label{S:Appendix}
%****************************************************************************
%****************************************************************************

Here we explain how, for an elliptic curve $E$, one can deduce explicit equations for the map $X^{-}_{E}(5) \to X(1)$ given knowledge of equations for $X_{E'}(5)_{\Qbar} \to X(1)_{\Qbar}$, where $E'$ is 2-isogenous to $E$ over $\Qbar$.\\

First we consider an abstract version of the problem: given a number field $K$ and a morphism $g\colon P^1_{\Qbar} \to P^1_{\Qbar}$ such that there exists an automorphism $\phi$ of $P^1_{\Qbar}$ such that $g \circ \phi$ is the base extension of a morphism $f\colon P^1_{\Q} \to P^1_{\Q}$, find such morphisms $\phi$ and $f$.

% First we consider an abstract version of the problem: given a number field $K$ and a commutative diagram
% \[
% \xymatrix{
% \P^1\ar[rd]_f\ar[rr]^{\phi_{/K}}&&\P^1\ar[ld]^{g_{/K}}\\
% &\P^1&
% }\]
% where $f$ is defined over $\Q$, $g$ and $\phi$ are defined over $K$, $\phi$ is an isomorphism, and one has equations for the map $g$, we would like to compute equations for the map $f$ up to $\Q$-automorphism of $\P^1$.

\begin{lemma}
\label{L:moebius}
  Let $\P^1_{\Qbar} \xrightarrow{\phi} \P^1_{\Qbar}$ be an automorphism and suppose there exist distinct $Q_1,Q_2,Q_3 \in \P^1(\Qbar)$ such that for every $i$ and for every $\sigma \in \GalQ$, $\phi(Q_i^{\sigma}) = (\phi(Q_i))^{\sigma}$. Then $\phi$ is defined over $\Q$.
\end{lemma}

\begin{proof}
  Representing $\phi$ as a M\"obius transformation $\frac{a_1z + a_2}{a_3z+a_4}$, one can, after scaling by a non-zero $a_i$, solve for the coefficients $a_i$ in terms of coordinates of the points $Q_j$ and $\phi(Q_j)$. It is then easy to see that $a_i^{\sigma} = a_i$ for all $\sigma \in \GalQ$ and thus $a_i \in \Q$.
\end{proof}

One can solve the abstract problem in the following situation: let $Q_1,Q_2,Q_3 \in \P^1(\Qbar)$ be 3 distinct points and suppose there exists an automorphism $\phi'\colon \P^1_{K} \to \P^1_{K}$ such that $\phi'\circ \phi$ and each $Q_i$ satisfy the hypothesis of lemma \ref{L:moebius}. Then $\phi'\circ \phi$ is defined over $\Q$. Setting $g' = g \circ (\phi')^{-1} $, we get a commutative diagram
\[
\xymatrix{
\P^1\ar[rrd]_f\ar[rr]^{\phi_{/K}}&&\P^1\ar[d]^{g_{/K}}\ar[rr]^{\phi'_{/K}}&& \P^1\ar[dll]^{g'}\\
&&\P^1&&
}\]
Since one has $f = g' \circ \phi' \circ \phi$, we conclude that $g'$ is defined over $\Q$ and differs from $f$ by the $\Q$-automorphism $\phi^{-1} \circ (\phi')^{-1}$.\\

Now suppose $E$ is given by the equation $y^2 = f(x)$. Let $K = \Q(E[2])$ be the splitting field of $f(x)$.  Then the 2-torsion points of $E$ are defined over $K$. Let $\{O, P_1, P_2, P_3\}$ be the 2-torsion points of $E$ and define $E_i$ to be $E/\langle P_i \rangle$, with $E \xrightarrow{\psi_i} E_i$ the quotient map and $\hat{\psi_i}$ the dual isogeny.

As in \cite{pss}*{4.4} the isogeny $E \xrightarrow{\psi_1} E_1$ changes the Weil pairing by $2$ (since for a 2-isogeny $\psi$, $\langle \psi (P), \psi (Q) \rangle = \langle \hat{\psi} \circ \psi (P), Q \rangle = \langle 2 P, Q \rangle = \langle P, Q \rangle^2$, where the first equality is \cite{Silverman:AEC}*{III.8.2}). 
In particular, if $E'$ is an elliptic curve and $E'[5] \cong E[5]$ is an anti-symplectic isomorphism (so that the map induced by the Weil pairing is $\zeta \mapsto \zeta^2$) then the composition $E'[5] \cong E[5] \xrightarrow{\psi_1} E_1[5]$ (where $\psi_1$ is the restriction to $E[5]$ of $\psi_1$) is symplectic.
This induces an isomorphism $\phi \colon X^{-}_{E}(5)_K \cong X_{E_1}(5)_K$ over $X(1)_K$, which is the situation of the above discussion.  

% p.10, footnote 1:
% Let psi=psi_1.
% I think for P,Q in E[5] one has
%   < psi P, psi Q > = < psi^dual psi P, Q >   (Silverman AEC III.8.2)
%                    = < 2 P, Q >
%                    = < P, Q >^2.

To apply this, let $Q_i \in X^-_E(5)(K)$ be three points induced by the isogenies $\psi_i$. By computing $j(E_i)$, it is easy to compute explicit points on $X_{E_1}(5)(K)$ which represent $\phi(Q_i)$. Let $\{r_1,r_2,r_3\}$ be the roots of $f(x)$. Since $r_i$ is the $x$-coordinate of the 2-torsion point $P_i \in E(K)$, the Galois action on the ordered set $\{r_1,r_2,r_3\}$ agrees with the action on $\{E_1,E_2,E_3\}$, and thus also on $\{Q_1,Q_2,Q_3\}$. Our hypothesis is that we have an explicit identification of $X_{E_1}(5)$ with $\P^1$. Thus, if we define $\phi'\colon \P^1 \to \P^1$ to be the map sending $\phi(Q_i)$ to $[r_1:1] \in \P^1(K)$, then the hypothesis of lemma \ref{L:moebius} is satisfied for the composition $\phi' \circ \phi$.

One can explicitly compute $E_i$, $j(E_j)$, $\phi(Q_i) \in X_{E_1}(K)$, and the map $\phi'$; equations for the map $j_{E_1}\colon X_{E_1}(5) \to X(1)$ are computed in \cite{RubinS:modpFamilies}. The composition $j_{E_1} \circ (\phi')^{-1}$ is thus an explicitly computable model for the map $X^-_E(5) \to X(1)$. Magma code doing all of this explicitly is available at \cite{me:2310Transcript}.

\begin{remark}
  Now let $E$ be given by the equation $y^2 = x^3 + ax + b$. In \cite{RubinS:modpFamilies} the equations for the map $X_E(5) \to X(1)$ are given as a function of the coefficients $a$ and $b$ of $E$. It it clear that the technique of this appendix can be refined to do the same for the map $X^{-}_{E}(5) \to X(1)$, since all of the numbers constructed (e.g. the $j$-invariants of the 2-isogenous curves $E_i$) depend algebraically on $a$ and $b$. However writing out the resulting equations would double the length of this paper and is thus omitted.
\end{remark}

%\bibliography{/Users/davidmbrownjr/Documents/notes/jabref/master.bib}

\begin{bibdiv}
\begin{biblist}

\bib{Magma}{article}{
      author={Bosma, Wieb},
      author={Cannon, John},
      author={Playoust, Catherine},
       title={The {M}agma algebra system. {I}. {T}he user language},
        date={1997},
        ISSN={0747-7171},
     journal={J. Symbolic Comput.},
      volume={24},
      number={3-4},
       pages={235\ndash 265},
        note={Computational algebra and number theory (London, 1993)},
      review={\MR{MR1484478}},
}

\bib{me:2310Transcript}{misc}{
      author={Brown, David},
       title={Electronic transcript of computations for the paper `primitive
  integral solutions to $x^2 + y^3 = z^{10}$'},
         url={http://www.math.berkeley.edu/\textasciitilde brownda/},
        note={Available at \url{http://www.math.berkeley.edu/\textasciitilde
  brownda/}. (Also attached at the end of the tex file.)},
}

\bib{Bruin:twoCubes}{incollection}{
      author={Bruin, Nils},
       title={On powers as sums of two cubes},
        date={2000},
   booktitle={Algorithmic number theory ({L}eiden, 2000)},
      series={Lecture Notes in Comput. Sci.},
      volume={1838},
   publisher={Springer},
     address={Berlin},
       pages={169\ndash 184},
      review={\MR{MR1850605 (2002f:11029)}},
}

\bib{Bruin:ellChab}{article}{
      author={Bruin, Nils},
       title={Chabauty methods using elliptic curves},
        date={2003},
        ISSN={0075-4102},
     journal={J. Reine Angew. Math.},
      volume={562},
       pages={27\ndash 49},
      review={\MR{MR2011330 (2004j:11051)}},
}

\bib{Cremona:modularBook}{book}{
      author={Cremona, J.~E.},
       title={Algorithms for modular elliptic curves},
     edition={Second},
   publisher={Cambridge University Press},
     address={Cambridge},
        date={1997},
        ISBN={0-521-59820-6},
      review={\MR{MR1628193 (99e:11068)}},
}

\bib{DiamondI:modularCurves}{incollection}{
      author={Diamond, Fred},
      author={Im, John},
       title={Modular forms and modular curves},
        date={1995},
   booktitle={Seminar on {F}ermat's {L}ast {T}heorem ({T}oronto, {ON},
  1993--1994)},
      series={CMS Conf. Proc.},
      volume={17},
   publisher={Amer. Math. Soc.},
     address={Providence, RI},
       pages={39\ndash 133},
      review={\MR{MR1357209 (97g:11044)}},
}

\bib{Edwards:235}{article}{
      author={Edwards, Johnny},
       title={A complete solution to {$X\sp 2+Y\sp 3+Z\sp 5=0$}},
        date={2004},
        ISSN={0075-4102},
     journal={J. Reine Angew. Math.},
      volume={571},
       pages={213\ndash 236},
      review={\MR{MR2070150 (2005e:11035)}},
}

\bib{Ellenberg:finiteFlatness}{article}{
      author={Ellenberg, Jordan~S.},
       title={Finite flatness of torsion subschemes of {H}ilbert-{B}lumenthal
  abelian varieties},
        date={2001},
        ISSN={0075-4102},
     journal={J. Reine Angew. Math.},
      volume={532},
       pages={1\ndash 32},
      review={\MR{MR1817501 (2003c:11062)}},
}

\bib{FreyM:arithmeticModular}{article}{
      author={Frey, Gerhard},
      author={M{\"u}ller, Michael},
       title={Arithmetic of modular curves and applications},
        date={1999},
       pages={11\ndash 48},
      review={\MR{MR1672093 (2000a:11095)}},
}

\bib{Jones:tables}{article}{
      author={Jones, John},
       title={Tables of number fields with prescribed ramification},
         url={http://math.asu.edu/~jj/numberfields/},
}

\bib{Macintyre:definable}{article}{
      author={Macintyre, Angus},
       title={On definable subsets of {$p$}-adic fields},
        date={1976},
        ISSN={0022-4812},
     journal={J. Symbolic Logic},
      volume={41},
      number={3},
       pages={605\ndash 610},
      review={\MR{MR0485335 (58 \#5182)}},
}

\bib{Mazur:isogenies}{article}{
      author={Mazur, B.},
       title={Rational isogenies of prime degree (with an appendix by {D}.
  {G}oldfeld)},
        date={1978},
        ISSN={0020-9910},
     journal={Invent. Math.},
      volume={44},
      number={2},
       pages={129\ndash 162},
      review={\MR{MR482230 (80h:14022)}},
}

\bib{McMurdy:explicitParam}{incollection}{
      author={McMurdy, Ken},
       title={Explicit parametrizations of ordinary and supersingular regions
  of {$X\sb 0(p\sp n)$}},
        date={2004},
   booktitle={Modular curves and abelian varieties},
      series={Progr. Math.},
      volume={224},
   publisher={Birkh\"auser},
     address={Basel},
       pages={165\ndash 179},
      review={\MR{MR2058650 (2005e:11074)}},
}

\bib{McCallumP:chabautySurvey}{article}{
      author={Poonen, Bjorn},
      author={McCallum, William},
       title={On the method of {C}habauty and {C}oleman},
        date={2007},
         url={http://www-math.mit.edu/~poonen/papers/chabauty.pdf},
}

\bib{pss}{article}{
      author={Poonen, Bjorn},
      author={Schaefer, Edward~F.},
      author={Stoll, Michael},
       title={Twists of {$X(7)$} and primitive solutions to {$x\sp 2+y\sp
  3=z\sp 7$}},
        date={2007},
        ISSN={0012-7094},
     journal={Duke Math. J.},
      volume={137},
      number={1},
       pages={103\ndash 158},
      review={\MR{MR2309145 (2008i:11085)}},
}

\bib{RubinS:mod2}{article}{
      author={Rubin, K.},
      author={Silverberg, A.},
       title={Mod {$2$} representations of elliptic curves},
        date={2001},
        ISSN={0002-9939},
     journal={Proc. Amer. Math. Soc.},
      volume={129},
      number={1},
       pages={53\ndash 57},
      review={\MR{MR1694877 (2001c:11064)}},
}

\bib{RubinS:modpFamilies}{article}{
      author={Rubin, K.},
      author={Silverberg, A.},
       title={Families of elliptic curves with constant mod {$p$}
  representations},
        date={1995},
       pages={148\ndash 161},
      review={\MR{MR1363500 (96j:11078)}},
}

\bib{Siksek:chabauty}{misc}{
  author = {Samir Siksek},
  title = {Explicit {C}habauty over Number Fields},
  year = {2010},
  note = {arXiv:1010.2603},
  owner = {davidmbrownjr},
  timestamp = {2009.09.12}
}


% \bib{Siksek:chabauty}{article}{
%       author={Siksek, Samir},
%        title={Explicit {C}habauty over number fields},
%      journal={Preprint},
% }

\bib{Silverman:AEC}{article}{
      author={Silverman, Joseph~H.},
       title={The arithmetic of elliptic curves},
        date={2009},
      volume={106},
       pages={xx+513},
      review={\MR{MR2514094}},
}

\end{biblist}
\end{bibdiv}
%\input{2310Bib.tex}
\def\cprime{$'$}

\end{document}